\documentclass[12pt]{article}

\usepackage{amsmath,amssymb,amsbsy,amsfonts,amsthm,latexsym,
            amsopn,amstext,amsxtra,amscd,stmaryrd,cmll,mathabx}

\usepackage[mathscr]{eucal}

\usepackage{subfig,graphics,epsfig,color}

\begin{document}

\newtheorem{theorem}{Theorem}
\newtheorem{lemma}[theorem]{Lemma}
\newtheorem{proposition}[theorem]{Proposition}
\newtheorem{corollary}[theorem]{Corollary}

% CALLIGRAPHIC ALPHABET

\def\cA{\mathcal A}
\def\cB{\mathcal B}
\def\cC{\mathcal C}
\def\cD{\mathcal D}
\def\cE{\mathcal E}
\def\cF{\mathcal F}
\def\cG{\mathcal G}
\def\cH{\mathcal H}
\def\cI{\mathcal I}
\def\cJ{\mathcal J}
\def\cK{\mathcal K}
\def\cL{\mathcal L}
\def\cM{\mathcal M}
\def\cN{\mathcal N}
\def\cO{\mathcal O}
\def\cP{\mathcal P}
\def\cQ{\mathcal Q}
\def\cR{\mathcal R}
\def\cS{\mathcal S}
\def\cU{\mathcal U}
\def\cT{\mathcal T}
\def\cV{\mathcal V}
\def\cW{\mathcal W}
\def\cX{\mathcal X}
\def\cY{\mathcal Y}
\def\cZ{\mathcal Z}

% SCRIPT ALPHABET

\def\sA{\mathscr A}
\def\sB{\mathscr B}
\def\sC{\mathscr C}
\def\sD{\mathscr D}
\def\sE{\mathscr E}
\def\sF{\mathscr F}
\def\sG{\mathscr G}
\def\sH{\mathscr H}
\def\sI{\mathscr I}
\def\sJ{\mathscr J}
\def\sK{\mathscr K}
\def\sL{\mathscr L}
\def\sM{\mathscr M}
\def\sN{\mathscr N}
\def\sO{\mathscr O}
\def\sP{\mathscr P}
\def\sQ{\mathscr Q}
\def\sR{\mathscr R}
\def\sS{\mathscr S}
\def\sU{\mathscr U}
\def\sT{\mathscr T}
\def\sV{\mathscr V}
\def\sW{\mathscr W}
\def\sX{\mathscr X}
\def\sY{\mathscr Y}
\def\sZ{\mathscr Z}

% FRAKTUR ALPHABET

\def\fA{\mathfrak A}
\def\fB{\mathfrak B}
\def\fC{\mathfrak C}
\def\fD{\mathfrak D}
\def\fE{\mathfrak E}
\def\fF{\mathfrak F}
\def\fG{\mathfrak G}
\def\fH{\mathfrak H}
\def\fI{\mathfrak I}
\def\fJ{\mathfrak J}
\def\fK{\mathfrak K}
\def\fL{\mathfrak L}
\def\fM{\mathfrak M}
\def\fN{\mathfrak N}
\def\fO{\mathfrak O}
\def\fP{\mathfrak P}
\def\fQ{\mathfrak Q}
\def\fR{\mathfrak R}
\def\fS{\mathfrak S}
\def\fU{\mathfrak U}
\def\fT{\mathfrak T}
\def\fV{\mathfrak V}
\def\fW{\mathfrak W}
\def\fX{\mathfrak X}
\def\fY{\mathfrak Y}
\def\fZ{\mathfrak Z}

% BLACKBOARD BOLD

\def\C{{\mathbb C}}
\def\F{{\mathbb F}}
\def\K{{\mathbb K}}
\def\L{{\mathbb L}}
\def\N{{\mathbb N}}
\def\Q{{\mathbb Q}}
\def\R{{\mathbb R}}
\def\Z{{\mathbb Z}}

% SOME STANDARD DEFINITIONS

\def\eps{\varepsilon}
\def\mand{\qquad\mbox{and}\qquad}
\def\mor{\qquad\mbox{or}\qquad}
\def\\{\cr}
\def\({\left(}
\def\){\right)}
\def\[{\left[}
\def\]{\right]}
\def\<{\langle}
\def\>{\rangle}
\def\fl#1{\left\lfloor#1\right\rfloor}
\def\rf#1{\left\lceil#1\right\rceil}
\def\le{\leqslant}
\def\ge{\geqslant}
\def\ds{\displaystyle}
\def\lcm{\text{\rm lcm}}

\def\xxx{\vskip5pt\hrule\vskip5pt}
\def\imhere{ \xxx\centerline{\sc I'm here}\xxx }

\newcommand{\comm}[1]{\marginpar{
\vskip-\baselineskip \raggedright\footnotesize
\itshape\hrule\smallskip#1\par\smallskip\hrule}}

% SPECIAL DEFINITIONS FOR THIS PAPER

\def\person#1#2{
\begin{minipage}{\textwidth}
\begin{flushleft} \large
#1
\end{flushleft}
\end{minipage}
\begin{minipage}{\textwidth}
\begin{flushright}
#2
\end{flushright}
\end{minipage} \vskip1cm}

%%%%%%%%%%%%%%%%%%%%%%%%%%%%%%%%%%%%%%%%%
%%%%%%%%%%  PAPER STARTS HERE  %%%%%%%%%%
%%%%%%%%%%%%%%%%%%%%%%%%%%%%%%%%%%%%%%%%%

\begin{titlepage}
\begin{center}

\textbf{\LARGE Self-intersections} \\[0.4cm]

\textbf{\LARGE of the Riemann zeta function} \\[0.4cm]

\textbf{\LARGE on the critical line} \\[1.5cm]

\person{William \textsc{Banks}}
{Department of Mathematics \\
University of Missouri \\
Columbia, MO 65211, USA \\
{\tt bbanks@math.missouri.edu}}

\person{Victor \textsc{Castillo-Garate}}
{Department of Mathematics \\
University of Missouri \\
Columbia, MO 65211, USA \\
{\tt vactn7@mail.missouri.edu}}

\person{Luigi \textsc{Fontana}}
{Dipartimento di Matematica ed Applicazioni \\
Universit\'a di Milano--Bicocca \\
Via Cozzi, 53 \\
20125 Milano, ITALY \\
{\tt luigi.fontana@unimib.it}}

\person{Carlo \textsc{Morpurgo}}
{Department of Mathematics \\
University of Missouri \\
Columbia, MO 65211 USA \\
{\tt morpurgoc@missouri.edu}}

\end{center}

\end{titlepage}

\begin{abstract}
We show that the Riemann zeta function $\zeta$
has only countably many self-intersections
on the critical line, i.e., for all but countably many
$z\in\C$ the equation $\zeta(\tfrac12+it)=z$ has at most
one solution $t\in\R$. More generally, we prove that if
$F$ is analytic in a complex neighborhood
of $\R$ and locally injective on $\R$, then either the set
$$
\{(a,b)\in\R^2:a\ne b\text{~and~}F(a)=F(b)\}
$$
is countable, or the image $F(\R)$ is a loop in $\C$.
\end{abstract}

\bigskip\bigskip

\section{Introduction}
\label{sec:intro}

In the half-plane $\{s\in\C:\sigma>1\}$
the \emph{Riemann zeta function} is defined by
the equivalent expressions
$$
\zeta(s):=\sum_{n=1}^\infty n^{-s}
=\prod_{p\text{~prime}}(1-p^{-s})^{-1}.
$$
In the extraordinary memoir of Riemann~\cite{Riemann} it is shown
that $\zeta$ extends to a meromorphic function on the entire complex plane
with its only singularity being a simple pole at $s=1$, and it satisfies the
functional equation relating its values at $s$ and $1-s$.
The Riemann hypothesis asserts that every non-real zero of $\zeta$
lies on the \emph{critical line}
$$
\sL:=\{s\in\C:\sigma=\tfrac12\}.
$$
By a \emph{self-intersection of $\zeta$ on the critical line} we mean an element of
the set
$$
\{(s_1,s_2)\in\sL^2:s_1\ne s_2\text{~and~}\zeta(s_1)=\zeta(s_2)\}.
$$
Our aim in the present paper is to prove that this set is countable.

\begin{theorem}
\label{thm:one}
The Riemann zeta function has only countably many
self-intersections on the critical line.
\end{theorem}

In other words, the equation $\zeta(\tfrac12+it)=z$ has at most
one solution $t\in\R$ for all but countably many $z\in\C$.  This
complements the fact that $\zeta(\tfrac12+it)=z$ has at least
two solutions $t\in\R$ for infinitely many $z\in\C$, which follows from
a recent result of Banks and Kang~\cite[Theorem~1.2]{BanksKang}.  Moreover, it has
been conjectured in \cite{BanksKang} that $\zeta(\tfrac12+it)=z$
has no more than two solutions $t\in\R$ for every nonzero $z\in\C$;
our Theorem~\ref{thm:one} makes it clear that there are at most countably
many counterexamples to this conjecture.

There are two main ingredients in the proof of Theorem~\ref{thm:one}.
The first is that the curve $f(t):=\zeta(\frac12+it)$ is
\emph{locally injective on $\R$}, i.e., for every $t\in\R$ there is an
open real interval containing $t$ on which $f$ is one-to-one;
this is Proposition~\ref{prop:hairy} of~\S\ref{sec:Riemann}.
A result of Levinson and Montgomery~\cite{LevMont}
guarantees the local injectivity of $f$
around points $t$ for which $f(t)\ne 0$, and in Proposition~\ref{prop:hairy} the
local injectivity of $f$ is also established around points $t$ with $f(t)=0$.

The second ingredient in the proof of Theorem~\ref{thm:one} is 
the following general result about self-intersections of locally
injective analytic curves.

\begin{theorem}
\label{thm:two}
Let $F$ be a function which is analytic in a complex neighborhood of the
real line~$\R$, and suppose that $F$ is locally injective on $\R$.
If the set of self-intersections
$$
\{(a,b)\in\R^2:a\ne b\text{~and~}F(a)=F(b)\}
$$
is uncountable, then $F(\R)$ is a loop in $\C$.
\end{theorem}

By a \emph{loop in $\C$} we mean the image of a continuous map
$L:[\gamma,\delta]\to\C$ such that $L(\gamma)=L(\delta)$.
Since every loop is compact, Theorem~\ref{thm:two}
applied with $F:=f$ immediately implies Theorem~\ref{thm:one}
in view of the fact that $\zeta$ is unbounded on
the critical line (see, for example, \cite[Theorem~8.12]{Titch}). 

The proof of Theorem~\ref{thm:two} in \S\ref{sec:sharing}
relies on intersection properties of analytic curves that
were first discovered by Markushevich~\cite{Mark}
and were later extended by Mohon$'$ko~\cite{Moho1,Moho2};
see Proposition~\ref{prop:mohonko} of \S\ref{sec:sharing}
and the remarks that follow.  We believe that the statement
of Theorem~\ref{thm:two} is new and may be of independent interest.

\section{Local injectivity}
\label{sec:Riemann}

\begin{proposition}
\label{prop:hairy}
The curve
\begin{equation}
\label{eq:fdefn}
f(t):=\zeta(\tfrac12+it)\qquad(t\in\R)
\end{equation}
is locally injective on $\R$.
\end{proposition}

\begin{proof}
For every $a\in\R$, let $\Sigma_a$ denote the collection of
open intervals $\cI$ in~$\R$ that contain $a$.
For any fixed $a$ we must show that $f$ is one-to-one
on an interval $\cI\in\Sigma_a$.

In the case that $f(a)\ne 0$
we use a result of Levinson and Montgomery~\cite{LevMont}
which states that $\zeta(s)=0$ whenever $\zeta'(s)=0$ and
$\sigma=\tfrac12$.  As $f(a)\ne 0$ we have $f'(a)\ne 0$,
hence $f$ is locally invertible in a complex neighborhood of $a$;
in particular, $f$ is one-to-one on some interval $\cI\in\Sigma_a$.

Let $\sZ$ denote the set of all zeros of $f$.
If $t\not\in\sZ$, we define
$$
\vartheta(t):=\arg f(t)=\Im\log\zeta(\tfrac12+it)
$$
by continuous variation of the argument from $2$ to $2+it$ to $\tfrac12+it$,
starting with the value $0$, and we denote by $N(t)$
the number of zeros $\rho=\beta+i\gamma$ of $\zeta(s)$
in the rectangle $0<\beta<1$, $0<\gamma<t$.  Then
$$
\vartheta(t)=\pi\cdot(N(t)-1)+g(t)\qquad(t\not\in\sZ),
$$
where
$$
g(t):=-\arg\Gamma\(\frac14+\frac{it}{2}\)+\frac{t}{2}\log\pi\qquad(t\in\R);
$$
see, for example, Montgomery and Vaughan~\cite[Theorem~14.1]{MontVau}.
Then
\begin{equation}
\label{eq:gp}
g'(t)=-\frac{1}{2}\,\Re\,\frac{\Gamma'}{\Gamma}\(\frac 14+\frac{it}{2}\)
+\frac{1}{2}\log\pi\qquad(t\in\R),
\end{equation}
and from the well known relation $(\Gamma'/\Gamma)'(s)=\sum_{n=0}^\infty(n+s)^{-2}$
we see that
$$
g''(t)=-16t\sum_{n=0}^\infty\frac{4n+1}{((4n+1)^2+4t^2)^2}\qquad(t\in\R).
$$
Thus, if $\Theta:=6.2898\cdots$ is the unique positive root of
the function defined by the right side of \eqref{eq:gp},
it is easy to see that $g$ is strictly decreasing at any $t\in\R$ with $|t|>\Theta$.

We now consider the case that $f(a)=0$, i.e., $a\in\sZ$.  Since
$\sZ$ is a discrete subset of $\R$, there exists an interval
$\cI\in\Sigma_a$ for which $\cI\cap\sZ=\{a\}$.
Using a shorter interval~$\cI$, if necessary, and
noting that $|a|\ge 14.1347\cdots>\Theta$, we can also ensure
that $g$ is strictly decreasing on $\cI$, and that
$|g(t)-g(a)|<1$ for all $t\in\cI$.  In particular,
$|g(t_1)-g(t_2)|<2$ for all $t_1,t_2\in\cI$.

Now suppose that $f(t_1)=f(t_2)$ with $t_1,t_2\in\cI$.
If $f(t_1)=f(t_2)=0$, then $t_1=t_2$ ($=a$) since $\cI\cap\sZ=\{a\}$.
If $f(t_1)=f(t_2)\ne 0$ we have
$$
0=\vartheta(t_1)-\vartheta(t_2)=\pi(N(t_1)-N(t_2))+g(t_1)-g(t_2),
$$
and therefore $g(t_1)-g(t_2)\in(-2,2)\cap\pi\Z=\{0\}$, i.e., $g(t_1)=g(t_2)$.
Since $g$ is strictly decreasing on $\cI$, we again have $t_1=t_2$.
This shows that $f$ is one-to-one on $\cI$, and the proof is complete.
\end{proof}

\section{Analytic functions sharing a curve}
\label{sec:sharing}

As in \S\ref{sec:Riemann}, for every $a\in\R$ we use $\Sigma_a$ to
denote the collection of open real intervals that contain~$a$. 

For each $\cI\in\Sigma_0$
we put $\cI^+:=\cI\cap[0,\infty)$ and $\cI^-:=\cI\cap(-\infty,0]$.
Given a function $F$ defined in a complex neighborhood of zero, we say that
\begin{itemize}
\item \emph{the curve of $F$ bounces back at zero} if 
there are arbitrarily short intervals $\cI\in\Sigma_0$
such that $F(\cI^+)=F(\cI^-)$.
\end{itemize}

Next, given two functions $F,G$
defined in a complex neighborhood of zero, we say that
\begin{itemize}
\item $F,G$ \emph{share a curve around zero}
if there are arbitrarily short intervals $\cI,\cJ\in\Sigma_0$
for which $F(\cI)=G(\cJ)$;
\item $F,G$ \emph{roughly agree near zero}
if there are two sequences of \emph{nonzero} real numbers,
$(u_k)_{k=1}^\infty$ and $(v_k)_{k=1}^\infty$,
such that $u_k\to 0$ and $v_k\to 0$
as $k\to\infty$, and $F(u_k)=G(v_k)$ for all $k\ge 1$;
\item $F,G$ \emph{roughly agree to the right of zero}
if there are two sequences of \emph{positive} real numbers,
$(u_k)_{k=1}^\infty$ and $(v_k)_{k=1}^\infty$, such that
$u_k\to 0^+$ and $v_k\to 0^+$ as $k\to\infty$, and $F(u_k)=G(v_k)$ for all $k\ge 1$.
\end{itemize}

\begin{proposition}
\label{prop:mohonko}
Suppose that $F,G$ are non-constant and analytic in some
neighborhood of zero, and that neither curve bounces back at zero.
Then the following statements are equivalent:
\begin{itemize}

\item[$(i)$] $F,G$ roughly agree near zero;

\item[$(ii)$] $F,G$ share a curve around zero.
\end{itemize}
\end{proposition}

\medskip\noindent{\sc Remarks:} With a few modifications to the
proof given below, one can establish the more general statement:
If $F,G$ are non-constant and analytic in a neighborhood of zero
and roughly agree to the right (or left) of zero, then either
$F,G$ share a curve around zero, or at least one of the two curves bounces
back at zero.  After proving Proposition~\ref{prop:mohonko} 
we discovered that the result is essentially contained
in the (untranslated) work of Mohon$'$ko~\cite{Moho1,Moho2}, who
extended earlier results of Markushevich~\cite[Vol.\ III, Theorem 7.20]{Mark} from regular
analytic curves to arbitrary analytic curves.  Here, we present our
own proof for the sake of completeness and for the convenience of
the reader.

\medskip We begin with the observation that Proposition~\ref{prop:mohonko}
is a consequence of the following statement.

\begin{lemma}
\label{lem:main2}
If $F,G$ satisfy the hypotheses of Proposition~\ref{prop:mohonko}
and roughly agree to the right of zero, then
they share a curve around zero.
\end{lemma}

Indeed, suppose the lemma has already been established.
The implication $(ii)\Rightarrow(i)$ of Proposition~\ref{prop:mohonko} being clear,
we need only show that $(i)\Rightarrow(ii)$. To this end,
let $\sU:=(u_k)_{k=1}^\infty$
and $\sV:=(v_k)_{k=1}^\infty$ be sequences of nonzero
real numbers such that $u_k\to 0$ and $v_k\to 0$ as $k\to\infty$,
and $F(u_k)=G(v_k)$ for all $k\ge 1$.  Replacing $\sU$ with a
subsequence, if necessary, we can assume that all of the numbers
in $\sU$ are of the same sign, and likewise for $\sV$;
that is, for some $\eps_\sU,\eps_\sV\in\{\pm 1\}$ we
have $u_k=\eps_\sU|u_k|$ and $v_k=\eps_\sV|v_k|$ for all $k\ge 1$.
Now, the functions $F_1,G_1$ defined by
$$
F_1(z):=F(\eps_\sU z)\mand G_1(z):=G(\eps_\sV z)
$$
are non-constant and analytic in a neighborhood of zero,
and neither the curve of $F_1$ nor that of $G_1$
bounces back at zero.  Also, $F_1,G_1$ roughly agree to the right
of zero since $|u_k|\to 0^+$ and $|v_k|\to 0^+$ as $k\to\infty$,
and
$$
F_1(|u_k|)=F(u_k)=G(v_k)=G_1(|v_k|)\qquad(k\ge 1).
$$
By Lemma~\ref{lem:main2}, it follows that $F_1,G_1$
share a curve around zero.  Hence, there are arbitrarily short intervals
$\cI,\cJ\in\Sigma_0$ such that
$$
F(\eps_\sU\cI)=F_1(\cI)=G_1(\cJ)=G(\eps_\sV\cJ).
$$
Since the intervals $\eps_\sU\cI,\eps_\sV\cJ$ lie in $\Sigma_0$ and are
arbitrarily short, it follows that $F,G$ share a curve around zero.
Thus, $(i)\Rightarrow(ii)$ as required.

\begin{proof}[Proof of Lemma~\ref{lem:main2}]
Let $(u_k)_{k=1}^\infty$ and $(v_k)_{k=1}^\infty$ be two sequences of positive
real numbers tending to zero such that $F(u_k)=G(v_k)$ for all $k\ge 1$.
Clearly, $F(0)=G(0)$, and we can assume $F(0)=G(0)=0$ without
loss of generality.  Since the functions $F,G$ are non-constant and
analytic, we can write
$$
F(z)=a_mz^m+a_{m+1}z^{m+1}+\cdots\mand
G(z)=b_nz^n+b_{n+1}z^{n+1}+\cdots
$$
for all $z$ in a neighborhood of zero, where
$m,n\ge 1$  and $a_mb_n\ne 0$.
As neither $F$ nor $G$ is identically zero, we have $F(u_k)=G(v_k)\ne 0$ for all large $k$.
Noting that $v_k^n/u_k^m>0$ and
$$
\frac{b_n}{a_m}
\frac{v_k^n}{u_k^m}=
\frac{F(u_k)}{a_mu_k^m}\frac{b_nv_k^n}{G(v_k)}
=\frac{1+O(u_k)}{1+O(v_k)}\to 1\qquad(k\to\infty),
$$
it follows that $b_n/a_m>0$; hence, replacing $F,G$ with 
$a_m^{-1}F$, $a_m^{-1}G$ we can assume without loss of generality
that $a_m=1$ and $b_n>0$.

Now put
\begin{equation}
\label{eq:m1n1def}
\ell:=\lcm[m,n],\qquad m_1:=\frac{\ell}{m}\,,\qquad n_1:=\frac{\ell}{n}\,,
\end{equation}
and write
\begin{align*}
F_1(z)&:=F(z^{m_1})=z^\ell+\cdots=z^\ell(1+F_2(z)),\\
G_1(z)&:=G(z^{n_1})=b_nz^\ell+\cdots=(c_nz)^\ell(1+G_2(z)),
\end{align*}
where $F_2,G_2$ are analytic in a neighborhood of zero,
with $F_2(0)=G_2(0)=0$, and $c_n=b_n^{1/\ell}>0$. The functions
\begin{align*}
F_3(z)&:=z\cdot\exp\big(\ell^{-1}\log(1+F_2(z))\big),\\
G_3(z)&:=c_n z\cdot\exp\big(\ell^{-1}\log(1+G_2(z))\big),
\end{align*}
are well-defined and analytic near zero, where $\log$ denotes the principal branch
of the logarithm, and since $F'_3(0)=1\ne 0$ and $G'_3(0)=c_n\ne 0$ these functions
are invertible near zero.  Therefore, we can define $H=F_3^{-1}\circ G_3$ as
an (invertible) analytic function in some neighborhood of zero.  Now,
since $F_3^\ell=F_1$ and $G_3^\ell=G_1$, for all large $k$ we have
$$
F_3(u_k^{1/m_1})^\ell=F_1(u_k^{1/m_1})=F(u_k)
=G(v_k)=G_1(v_k^{1/n_1})=G_3(v_k^{1/n_1})^\ell,
$$
and thus $F_3(u_k^{1/m_1})=\xi_k\cdot G_3(v_k^{1/n_1})$ holds with some
$\ell$-th root of unity $\xi_k$.  On the other hand, for all large $k$ the
relation
$$
\frac{F_3(u_k^{1/m_1})}{G_3(v_k^{1/n_1})}
=\frac{u_k^{1/m_1}}{c_nv_k^{1/n_1}}\exp\(\ell^{-1}
\log\(\frac{1+F_2(u_k^{1/m_1})}{1+G_2(v_k^{1/n_1})}\)\)
$$
implies that
$$
\arg\xi_k=(i\,\ell)^{-1}\log\(\frac{1+F_2(u_k^{1/m_1})}{\big|1+F_2(u_k^{1/m_1})\big|}\cdot
\frac{\big|1+G_2(v_k^{1/n_1})\big|}{1+G_2(v_k^{1/n_1})}\)\to 0
\qquad(k\to\infty);
$$
thus, for all sufficiently large $k$ we must have $\xi_k=1$,
$F_3(u_k^{1/m_1})=G_3(v_k^{1/n_1})$, and $H(v_k^{1/n_1})=u_k^{1/m_1}$,
and it follows that $H$ defines a real-valued
invertible real analytic function on some interval
$\cK\in\Sigma_0$.\footnote{To see why $H$ is real-valued on some
in $\Sigma_0$, observe
that $h(z):=H(z)-\overline{H(\overline{z})}$ is analytic near zero,
$h(v_k^{1/n_1})=0$ for all large $k$ since $u_k,v_k\in\R$, and
$v_k^{1/n_1}\to 0$ as $k\to\infty$, hence by the principle of analytic continuation
$h$ must vanish identically near zero.}

From the relation
$$
G_3(t)=F_3(H(t))\qquad(t\in\cK),
$$
we see that
\begin{equation}
\label{eq:gfreln}
G(t^{n_1})=G_1(t)=F_1(H(t))=F(H(t)^{m_1})\qquad(t\in\cK).
\end{equation}
We claim that $n_1$ is odd.  Indeed, suppose on the contrary
that $n_1$ is even.  From the definition \eqref{eq:m1n1def}
it follows that $\gcd(m_1,n_1)=1$, hence $m_1$ is odd in this case.
Since $H$ is real-valued and invertible, $H(0)=0$, and $H$ maps
$v_k^{1/n_1}$ to $u_k^{1/m_1}>0$ if $k$ is large, we have
$H(\cK^+)\subseteq[0,\infty)$ and $H(\cK^-)\subseteq(-\infty,0]$.
Thus, if $\eps>0$ is small enough so that the interval
$\cL:=(H(-\eps)^{m_1},H(\eps)^{m_1})\in\Sigma_0$
is contained in $\cK$, then $F(\cL^+)=F(\cL^-)$ since by \eqref{eq:gfreln}:
$$
F(H(t)^{m_1})=G(t^{n_1})=G((-t)^{n_1})=F(H(-t)^{m_1})\qquad(t\in[0,\eps)).
$$
But this means that the curve of $F$ bounces back at zero, which
contradicts our original hypothesis on $F$.  This contradiction
establishes our claim that $n_1$ is odd.  A similar argument shows
that $m_1$ is also odd.

Since $m_1$ and $n_1$ are both odd, the intervals
$\cI:=H(\cK)^{m_1}$, $\cJ:=\cK^{n_1}$, are both open and contain zero,
thus $\cI,\cJ\in\Sigma_0$. By \eqref{eq:gfreln} we have
$F(\cI)=G(\cJ)$.  Replacing $\cK$ by an arbitrarily short interval in $\Sigma_0$,
the intervals $\cI,\cJ$ become arbitrarily short as well, and therefore
$F,G$ share a curve around zero.
\end{proof}

The ideas presented above can be adapted as follows.
Given $a,b\in\R$ and a function $F$ defined in some
complex neighborhoods of $a$ and $b$, we say that
\begin{itemize}
\item \emph{$F$ shares curves around $a,b$} if there exist 
arbitrarily short intervals $\cI\in\Sigma_a$ and $\cJ\in\Sigma_b$
such that $F(\cI)=F(\cJ)$.
\end{itemize}
This is equivalent to the statement that $F_a,F_b$ share a curve around zero,
where for each $a\in\R$ we denote by $F_a$ the function given by $F_a(z)=F(z+a)$.
When viewed as a function of a real variable, if $F$ is one-to-one locally at
$a$ and $b$, then the curves of $F_a,F_b$ cannot bounce back at zero; hence,
the next statement is an easy consequence of Proposition~\ref{prop:mohonko}.

\begin{lemma}
\label{lem:share}
Fix $a,b\in\R$. Suppose that $F$ is analytic in
some neighborhoods of $a$ and $b$, and that there are
intervals in both $\Sigma_a$ and $\Sigma_b$ on which $F$ is
one-to-one.  Then the following statements are equivalent:
\begin{itemize}

\item[$(i)$] $F$ shares curves around $a,b$;

\item[$(ii)$] there are sequences $(u_k)_{k=1}^\infty$
and $(v_k)_{k=1}^\infty$ such that

\subitem $\circ~~~$ $u_k\ne a$ and $v_k\ne b$ for all $k\ge 1;$

\subitem $\circ~~~$ $u_k\to a$ and $v_k\to b$ as $k\to\infty;$

\subitem $\circ~~~$  $F(u_k)=F(v_k)$ for all $k\ge 1$.

\end{itemize}
\end{lemma}

For real numbers $\alpha<\beta$ we now assume that $F$
is analytic in a complex neighborhood of the interval $[\alpha,\beta]$,
and that $F$ is \emph{locally injective} at every point of $[\alpha,\beta]$,
i.e., for any $a\in[\alpha,\beta]$ there is an interval in $\Sigma_a$ on
which $F$ is one-to-one.  By a \emph{self-intersection of $F$ on $[\alpha,\beta]$}
we mean an element of
$$
\sS_{\alpha,\beta}:=\{(a,b)\in[\alpha,\beta]^2:a\ne b\text{~and~}F(a)=F(b)\}.
$$
Our goal is to understand the structure of this set.

\begin{lemma}
\label{lem:Scompact}
The set $\sS_{\alpha,\beta}$ is compact.
\end{lemma}

\begin{proof}
Since $[\alpha,\beta]^2$ is compact, it suffices to show that $\sS_{\alpha,\beta}$ is closed.

Let $(a,b)\in[\alpha,\beta]^2$ and suppose that
$(u_k,v_k)\to(a,b)$ as $k\to\infty$,
where $(u_k,v_k)\in\sS_{\alpha,\beta}$ for all $k\ge 1$.
Since $F(u_k)=F(v_k)$ for each $k$,
the continuity of~$F$ implies that $F(a)=F(b)$.

Assume that $a=b$.  Since $F$ is locally injective at $a$,
there is an interval $\cI\in\Sigma_a$ on which $F$ is
one-to-one.  On the other hand, for all large $k$
we have $u_k,v_k\in\cI$, $u_k\ne v_k$,
and $F(u_k)=F(v_k)$, which shows that $F$ is not one-to-one on $\cI$.
The contradiction implies that $a\ne b$, hence $(a,b)\in\sS_{\alpha,\beta}$.
\end{proof}

We now express $\sS_{\alpha,\beta}$ as a disjoint union
$\sS_{\alpha,\beta}^*\mathop{\cup}\limits^{\textbf{.}}\sS_{\alpha,\beta}^\circ$,
where $\sS_{\alpha,\beta}^*$ is the set of \emph{limit points} of $\sS_{\alpha,\beta}$,
and  $\sS_{\alpha,\beta}^\circ$ is the set of
\emph{isolated points} in $\sS_{\alpha,\beta}$. The next lemma
provides a useful characterization of the set $\sS_{\alpha,\beta}^*$.

\begin{lemma}
We have
\begin{equation}
\label{eq:businesscards}
\sS_{\alpha,\beta}^*=\{(a,b)\in[\alpha,\beta]^2:a\ne b
\text{~and $F$ shares curves around $a,b$}\}.
\end{equation}
\end{lemma}

\begin{proof}
Let $\sT$ denote the set on the right side of \eqref{eq:businesscards}. In view
of Lemma~\ref{lem:share}, the inclusion $\sT\subseteq\sS_{\alpha,\beta}^*$ is
clear. On the other hand, if $(a,b)$ is a limit point of $\sS_{\alpha,\beta}$
that does not lie in $\sT$, then one of the following statements is true:
\begin{itemize}
\item[$(i)$] there is a sequence $(u_k)_{k=1}^\infty$
such that $u_k\ne a$ for all $k\ge 1$,
$u_k\to a$ as $k\to\infty$, and $F(u_k)=F(b)$ for all $k\ge 1$;

\item[$(ii)$] there is a sequence $(v_k)_{k=1}^\infty$
such that $v_k\ne b$ for all $k\ge 1$,
$v_k\to b$ as $k\to\infty$, and $F(a)=F(v_k)$ for all $k\ge 1$.
\end{itemize}
However, since $F$ is analytic in a  neighborhood of $[\alpha,\beta]$,
either statement implies that $F$ is constant on the same neighborhood,
contradicting the local injectivity of $F$ on $[\alpha,\beta]$.  This shows that
$\sS_{\alpha,\beta}^*\setminus\sT=\varnothing$, and \eqref{eq:businesscards}
follows.
\end{proof}

\begin{lemma}
\label{lem:twisty}
For every $(a,b)\in\sS_{\alpha,\beta}^*$ there are intervals $\cI\in\Sigma_a$, $\cJ\in\Sigma_b$, 
and a continuous bijection $\phi_{a,b}:\cI\to\cJ$ such that
$\phi_{a,b}(a)=b$, and $(t,\phi_{a,b}(t))\in\sS_{\alpha,\beta}^*$ for all $t\in\cI$.  Moreover,
$\phi_{a,b}$ is a strictly increasing function of $t$.
\end{lemma}

\begin{proof}
For any $(a,b)\in\sS_{\alpha,\beta}^*$, there are arbitrarily short
open intervals $\cI\in\Sigma_a$ and $\cJ\in\Sigma_b$ such
that $F(\cI)=F(\cJ)$.  Since $a\ne b$ we can assume $\cI\cap\cJ=\varnothing$.
Since $F$ is locally injective, we can further assume $F$ is one-to-one
on $\cI$ and on~$\cJ$. Then $\phi_{a,b}$ is the map obtained via the composition
$$
\cI~~\mathop{\xrightarrow{\hspace*{0.9cm}}}\limits^{F\vert_\cI}~~F(\cI)=F(\cJ)
~~\mathop{\xrightarrow{\hspace*{0.9cm}}}\limits^{(F\vert_\cJ)^{-1}}~~\cJ
$$
Clearly, $(t,\phi_{a,b}(t))\in\sS_{\alpha,\beta}$ for all $t\in\cI$, and as
$\cI$ is open and $\phi_{a,b}$ is continuous,
every such $(t,\phi_{a,b}(t))$ is a limit point of $\sS_{\alpha,\beta}$.
This proves the first statement.

Observe that, since $F$ is one-to-one on $\cI$ and $\phi_{a,b}$ is continuous,
it follows that $\phi_{a,b}$ is either strictly increasing or strictly decreasing
on its interval of definition. Thus, to finish the proof it suffices to
show that the set
$$
\sS_{\alpha,\beta}':=\{(a,b)\in\sS_{\alpha,\beta}^*:
\phi_{a,b}\text{~is strictly decreasing on $\cI$}\}
$$
is empty. First, we claim that $\sS_{\alpha,\beta}'$ is closed,
hence compact by Lemma~\ref{lem:Scompact}.  Indeed,
if $(a,b)$ is any limit point of $\sS_{\alpha,\beta}'$, then $(a,b)$ lies in
the \emph{closed} set $\sS_{\alpha,\beta}^*$. Let $\cI,\cJ,\phi_{a,b}$ be defined
as before.  Since $(a,b)$ is a limit point, we can find $(u,v)\in\sS_{\alpha,\beta}'$
for which $u\in\cI$, $v\in\cJ$.  As $\phi_{u,v}$ is strictly decreasing
near $u$, the same is true of $\phi_{a,b}$, hence $\phi_{a,b}$ is
strictly decreasing on $\cI$.  This implies that $(a,b)\in\sS_{\alpha,\beta}'$,
and the claim follows.

Assume that $\sS_{\alpha,\beta}'\ne\varnothing$. Since $\sS_{\alpha,\beta}'$ is compact,
there exists $(a,b)\in\sS_{\alpha,\beta}'$ for which
\begin{equation}
\label{eq:easybutton}
|a-b|=\min\limits_{(u,v)\in\sS_{\alpha,\beta}'}|u-v|\ne 0.
\end{equation}
Interchanging $a$ and $b$, if necessary, we can assume that
$a<b$. Let $\cI,\cJ,\phi_{a,b}$ be defined as before.  Since $\phi_{a,b}$
is strictly decreasing on $\cI$, using the first statement of the lemma
it is easy to see that $(t,\phi_{a,b}(t))$ lies in $\sS_{\alpha,\beta}'$ for all
$t\in\cI$; in particular, if $u\in\cI$, $u>a$, and $v:=\phi_{a,b}(u)$,
then $(u,v)\in\sS_{\alpha,\beta}'$ with $v<b$.  But then we have
$|u-v|<|a-b|$, which contradicts \eqref{eq:easybutton}.
Hence, it must be the case that $\sS_{\alpha,\beta}'=\varnothing$,
and we are done.
\end{proof}

Given $a,b\in[\alpha,\beta]$, we write $a\cong b$ and say that
\emph{$a$ is $F$-equivalent to~$b$} if $a=b$ or $(a,b)\in\sS_{\alpha,\beta}^*$;
otherwise, we say that \emph{$a$ and $b$ are $F$-inequivalent}.
Using Lemma~\ref{lem:share} one verifies that this defines
an equivalence relation on $[\alpha,\beta]$.
We also write $a\rightsquigarrow b$ and say that \emph{$a$ is the immediate
$F$-predecessor of $b$} if $a<b$, $a\cong b$, and there is no
number $c$ such that $a<c<b$ and $c\cong a$.

It is important to note that if $a\cong b$ and $a<b$, then $b$
has an immediate $F$-predecessor $c$ that lies in the interval $[a,b)$;
this follows from the fact that the numbers in a fixed \text{$F$-equivalence}
class are isolated since $F$ is analytic and locally injective.  

\begin{proposition}
\label{prop:funddomain}
Suppose that $\gamma,\delta$ are $F$-equivalent numbers in $[\alpha,\beta]$,
and that $\gamma$ is the immediate $F$-predecessor of $\delta$. Then
\begin{itemize}
\item[$(i)$] every number in $[\alpha,\beta]$ is $F$-equivalent to some number
in $[\gamma,\delta)$;

\item[$(ii)$] distinct numbers in $[\gamma,\delta)$
are $F$-inequivalent.
\end{itemize}
\end{proposition}

\begin{proof}
To prove $(i)$, suppose on the contrary that there exists $\xi\in[\alpha,\beta]$
which is \text{$F$-inequivalent} to every number in $[\gamma,\delta)$.  To achieve
the desired contradiction, we examine three cases separately.

\bigskip\noindent{\sc Case 1}: $\xi\in[\gamma,\delta]$. 
Since $\gamma\cong\delta$, we have $\xi\in(\gamma,\delta)$,
so this case is not possible as $\xi$ is $F$-equivalent to itself.

\bigskip\noindent{\sc Case 2}: $\xi\in[\alpha,\gamma)$.  Under this assumption, the supremum
$$
\eta:=\sup\{\tau\in[\alpha,\gamma):\tau
\text{~is $F$-inequivalent to every number in $[\gamma,\delta)$}\}
$$
exists, and we have $\eta\in[\alpha,\gamma]$. First, we claim that
$\eta\not\cong\gamma$.  Assume on the contrary that
$\eta\cong\gamma\cong\delta$.  We cannot have $\eta=\xi$ since
$\xi$ is \text{$F$-inequivalent} to every number in $[\gamma,\delta)$;
therefore, $\eta>\xi\ge\alpha$. By Lemma~\ref{lem:twisty},
$t\cong\phi_{\eta,\delta}(t)$ for all $t$ in some open
interval $\cI\in\Sigma_{\eta}$.
Since $\phi_{\eta,\delta}$ is strictly increasing on $\cI$
and $\phi_{\eta,\delta}(\eta)=\delta$,
we see that \emph{every} $\tau\in[\alpha,\eta)$ which is sufficiently close to $\eta$ 
is $F$-equivalent to a number in $[\gamma,\delta)$, namely $\phi_{\eta,\delta}(\tau)$.
This contradicts the definition of $\eta$ and thereby establishes the claim.

Now let $(u_k)_{k=1}^\infty$ be a strictly decreasing sequence in the
interval $(\eta,\gamma)$, with $u_k\to\eta^+$ as $k\to\infty$.  From the
definition of $\eta$, each $u_k$ is $F$-equivalent to some
$v_k\in[\gamma,\delta)\subseteq[\gamma,\delta]$.  Since
$\sS_{\alpha,\beta}^*$ and $[\gamma,\delta]$ are compact,
choosing a subsequence (if necessary)
we can assume that $(u_k,v_k)\to(\eta,\lambda)\in\sS_{\alpha,\beta}^*$ with
some $\lambda\in[\gamma,\delta]$.
Moreover, such $\lambda$ lies in the open interval $(\gamma,\delta)$
since $\eta\not\cong\gamma$ as shown above.  However, using Lemma~\ref{lem:twisty} again,
it follows that any $t$ that is sufficiently close to $\eta$ is \text{$F$-equivalent}
to a number in $[\gamma,\delta)$, namely $\phi_{\eta,\lambda}(t)$, and this
again contradicts to the definition of $\eta$.
Therefore, the case $\xi\in[\alpha,\gamma)$ is not possible.

\bigskip\noindent{\sc Case 3}: $\xi\in(\delta,\beta]$.
This is also impossible by an argument similar to that
given in Case~2, using instead the definition
$$
\eta:=\inf\{\tau\in(\delta,\beta]:
\tau\text{~is $F$-inequivalent to every number in $[\gamma,\delta)$}\}.
$$
We omit the details.  This completes our proof of $(i)$.

Turning to the proof of $(ii)$, suppose on the contrary that $\mu\cong\theta$ and
$\gamma\le\mu<\theta<\delta$; without loss of generality, we can assume that
$\mu\rightsquigarrow\theta$.  Since
$\gamma\rightsquigarrow\delta$, the interval $(\gamma,\delta)$ 
contains no number that is $F$-equivalent to $\gamma$; therefore,
$\theta\not\cong\gamma$, so $\mu\not\cong\gamma$ and $\gamma<\mu$.

Now, fix an arbitrary number $\lambda\in(\theta,\delta)$.
Applying the previous result $(i)$ with the sets $[\mu,\theta)$ and $[\gamma,\lambda]$
in place of $[\gamma,\delta)$ and $[\alpha,\beta]$, respectively, we see that
every number in $[\gamma,\lambda]$ is $F$-equivalent to some number
in $[\mu,\theta)$; in particular, $\gamma$ is $F$-equivalent to
a number $\eta$ for which $\gamma<\mu\le\eta<\theta<\delta$.  But this
is impossible since $\gamma\rightsquigarrow\delta$, and the
contradiction completes the proof of $(ii)$.
\end{proof}

\begin{proof}[Proof of Theorem~\ref{thm:two}]
Suppose that $\sS$ is uncountable, and put
$$
\sS^*:=\{(a,b)\in\R^2:a\ne b\text{~and $F$ shares curves around $a,b$}\}.
$$
We claim that $\sS^*\ne\varnothing$.  Indeed, assuming that $\sS^*=\varnothing$,
we have $\sS_{-n,n}^*=\varnothing$ for every natural number $n$, and therefore
$\sS_{-n,n}=\sS_{-n,n}^\circ$ is \emph{finite} as it is a closed
and isolated subset of the compact set $[-n,n]^2$.
But then $\sS=\bigcup_{n\ge 1}\sS_{-n,n}$
is a countable union of finite sets, hence countable.
This contradiction yields the claim.

Now, since $\sS^*\ne\varnothing$, we can find distinct real numbers
$\gamma,\delta$ such that $\gamma\cong\delta$,
and we can assume that $\gamma\rightsquigarrow\delta$.
As $[\gamma,\delta]\subseteq[-n,n]$ for all large~$n$, and $F$ is
constant on every $F$-equivalence class, Proposition~\ref{prop:funddomain} implies
that $F([-n,n])=F([\gamma,\delta])$ for all large $n$.  Letting $n\to\infty$,
we see that $\sC=F([\gamma,\delta])$, which is a loop in $\C$ since
$F(\gamma)=F(\delta)$.
\end{proof}

\section{Concluding remarks}

It would be interesting to determine whether the appropriate analogue
of Theorem~\ref{thm:one} can be established for arbitrary automorphic
$L$-functions.  For many functions in the Selberg class,
including all Dirichlet $L$-functions, an analogue of
Voronin's Universality Theorem is available and thus one can
prove that those $L$-functions are unbounded on the
critical line. However, establishing the local injectivity of
such functions may be difficult.

With a little extra work, Theorem~\ref{thm:two} can be generalized
and applied to arbitrary analytic curves on $\R$ (or on an interval in $\R$).
For such functions~$F$, if the set of self-intersections is uncountable,
then the curve continues to wind in the same direction around a
loop (when $F$ is locally injective), or else it bounces back at
least once along a regular curve, which need not be a loop
nor even bounded.

\bigskip

\textbf{Acknowledgments.}
The first author would like to thank Roger Baker, David Cardon
and Roger Heath-Brown for encouraging conversations during the
early stages of this project.

\end{document}